\documentclass[11pt]{article}
\usepackage{amsmath, amsthm, amssymb,latexsym,mathscinet}
\usepackage{hyperref}
\usepackage{caption,subcaption,xcolor}
\usepackage{tikz,graphicx,cite,todonotes}
\usepackage{booktabs}
\usepackage{multirow}
\usepackage{setspace}
\usepackage{dsfont}
\usepackage{enumerate,enumitem}

\newtheorem{thm}{Theorem}[section]

\newtheorem{prop}[thm]{Proposition}

\newtheorem{clm}[thm]{Claim}

\newtheorem{ques}[thm]{Question}

\numberwithin{equation}{section}

\newcommand{\bp}{{\mathcal{P}}}
\newcommand{\iii}[6]{{[#1, #2]\times[#3, #4]\times[#5, #6]}}

\begin{document}

\title{Brick partition problems in three dimensions}

\author{Ilkyoo Choi\thanks{Supported by the Basic Science Research Program through the National Research Foundation of Korea (NRF) funded by the Ministry of Education (NRF-2018R1D1A1B07043049), and also by the Hankuk University of Foreign Studies Research Fund.
Department of Mathematics, Hankuk University of Foreign Studies, Yongin-si, Gyeonggi-do, Republic of Korea.
Corresponding author. 
\texttt{ilkyoo@hufs.ac.kr}
}
\and Minseong Kim\thanks{Gyeonggi Science High School for the Gifted, Suwon-si, Gyeonggi-do, Republic of Korea. 
\texttt{aloeflavor@gmail.com}
}
\and Kiwon Seo\thanks{Gyeonggi Science High School for the Gifted, Suwon-si, Gyeonggi-do, Republic of Korea. 
\texttt{seokw0891@naver.com}
}
}

\date\today
\maketitle

\begin{abstract}
A {\it $d$-dimensional brick} is a set $I_1\times \cdots \times I_d$ where each $I_i$ is an interval. 
Given a brick $B$, a {\it brick partition of $B$} is a partition of $B$ into bricks. 
A brick partition $\bp_d$ of a $d$-dimensional brick is {\it $k$-piercing} if every axis-parallel line intersects at least $k$ bricks in $\bp_d$. 
Bucic et al.~\cite{2019BuLiLoWa} explicitly asked the minimum size $p(d, k)$ of a $k$-piercing brick partition of a $d$-dimensional brick.
The answer is known to be $4(k-1)$ when $d=2$. 
Our first result almost determines $p(3, k)$.
Namely, we construct a $k$-piercing brick partition of a $3$-dimensional brick with $12k-15$ parts, which is off by only $1$ from the known lower bound. 
As a generalization of the above question, we also seek the minimum size $s(d, k)$ of a brick partition $\bp_d$ of a $d$-dimensional brick where each axis-parallel plane intersects at least $k$ bricks in $\bp_d$. 
We resolve the question in the $3$-dimensional case by determining $s(3, k)$ for all $k$.
\end{abstract}

\section{Introduction}\label{sec:intro}

A {\it $d$-dimensional brick} is a set $I_1\times \cdots \times I_d$ where each $I_i$ is an interval. 
Given a $d$-dimensional  brick $B$, a partition $\{B_1, \ldots, B_m\}$ of $B$ is a {\it brick partition of $B$} if each $B_i$ is a $d$-dimensional brick. 
(Note that we allow the parts to share boundaries, but no interior points.)
A brick partition $\{B_1, \ldots, B_m\}$ is {\it $k$-piercing} if every axis-parallel line intersects at least $k$ distinct $B_i$. 
The following question was explicitly formulated in~\cite{2019BuLiLoWa}:

\begin{ques}
For integers $d\geq 1$ and $k\geq 2$, what is the minimum size $p(d, k)$ of a $k$-piercing brick partition of a $d$-dimensional brick?
\end{ques}

The authors of~\cite{2019BuLiLoWa} consider discrete bricks (a product of sets of consecutive integers), but they remark that the question can be formalized in the continuous setting, which is the topic of this paper.

Elementary arguments show that $d2^{d-1}(k-2)+2^d\leq p(d, k)\leq k^d$.
The proof of the lower bound is the following: 
assume $\bp$ is a $k$-piercing brick partition of a $d$-dimensional brick $B$. 
Each edge of $B$ must be incident with at least $k$ bricks, and if a brick in $\bp$ is incident with two edges of $B$, then it must be incident with a corner of $B$ since $\bp$ is $k$-piercing.  
The lower bound follows since $B$ has $d2^{d-1}$ edges and $2^d$ corners.
The upper bound is given by simply partitioning the initial brick $B$ into $k$ parts along each dimension, which gives a $k$-piercing brick partition with $k^d$ bricks. 
An intriguing result in~\cite{2019BuLiLoWa} states that when $d$ is fixed, $p(d, k)$ is actually bounded above by a linear function of $k$, namely, $p(d, k)\leq 3.92^dk$. 

Exact values of $p(d, k)$ are known whenever $k=2$ or $d=2$. 
When $k=2$, the elementary bounds in the previous paragraph imply $p(d, 2)=2^d$. 
When $d=2$, the example in Figure~\ref{fig:brick2d} demonstrates $p(2, k)\leq 4(k-1)$, which matches the elementary lower bound in the aforementioned paragraph. 
Hence, $p(2, k)=4(k-1)$. 

\begin{figure}[h!]
  \centering
  \includegraphics[height=4.5cm,page=1]{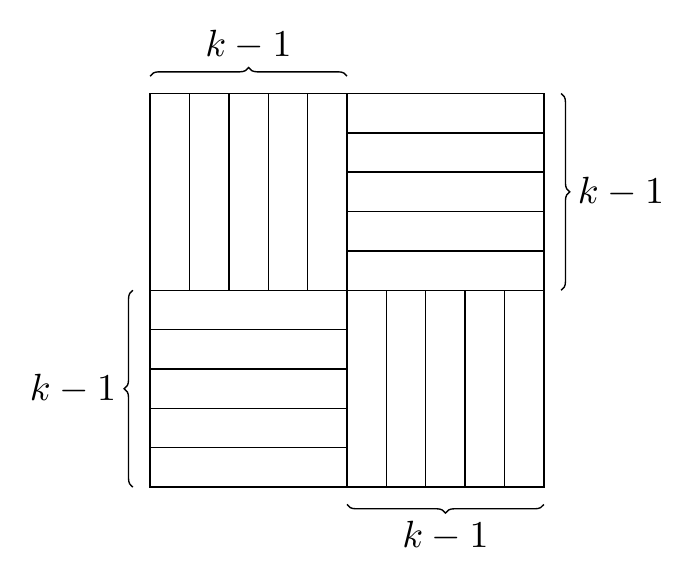}
  \caption{A $k$-piercing brick partition of a $2$-dimensional brick with only $4(k-1)$ bricks.}
  \label{fig:brick2d}
\end{figure}

For $k\geq 3$, our first result almost determines $p(3, k)$ by narrowing down the possible values to $12k-16$ and $12k-15$. 
Given a $3$-dimensional brick $B$, we explicitly construct a $k$-piercing brick partition of $B$  that uses only $12k-15$ bricks. 
See Figure~\ref{fig:brick3d-p}.
In other words, $p(3, k)\leq 12k-15$, which is off by only $1$ since the above elementary lower bound gives $12k-16$ when $d=3$.  

\begin{thm}\label{thm:piercing}
For an integer $k\geq 3$ and a $3$-dimensional brick $B$, there exists a $k$-piercing brick partition of $B$ with only $12k-15$ bricks.
In other words, $p(3, k)\leq 12k-15$.
\end{thm}

Since $k$-piercing concerns an intersection property of a brick partition with a $1$-dimensional object (line), we can generalize this concept to higher dimensions. 
We say a (hyper)plane $p$ is {\it axis-parallel} if each axis is either parallel or orthogonal to $p$. 
Hence, as a generalization of $k$-piercing, we introduce the notion of $k$-slicing:
a brick partition $\{B_1, \ldots, B_m\}$ is {\it $k$-slicing} if every axis-parallel $2$-dimensional plane intersects at least $k$ distinct $B_i$. 
We seek the minimum size $s(d, k)$ of a $k$-slicing brick partition of a $d$-dimensional brick. 

\begin{ques}
For integers $d\geq 2$ and $k\geq 2$, what is the minimum size $s(d, k)$ of a $k$-slicing brick partition of a $d$-dimensional brick?
\end{ques}

We resolve the above question for the first nontrivial dimension of by determining $s(3, k)$ for all values of $k$.
Except when $k=2$, the value of $s(3, k)=2k-1$. 
Namely, for $k\geq 3$, given a $3$-dimensional brick $B$, we explicitly construct a $k$-slicing brick partition of $B$  that uses only $2k-1$ bricks, and show that such a partition cannot be obtained with fewer bricks. 

\begin{thm}\label{thm:slicing}
For an integer $k\geq 3$, $s(3, k)=2k-1$, and $s(3, 2)=4$.  
\end{thm}

Note that in three dimensions, the only types of intersection properties for brick partitions are $k$-piercing and $k$-slicing. 
Other types of intersection properties exist for higher dimensions, as well as for other types of partitions such as box partitions; we direct the interested readers to~\cite{2002AlBoHoKl,2019BuLiLoWa,2006KiPr,2004GrKiPr,2005AhYu,1993Kleitman,2008KiPr,2019Holzman} for more results of this flavor. 

In Section~\ref{sec:piercing}, we provide a construction of a $k$-piercing brick partition of a $3$-dimensional brick that uses only $12k-15$ bricks. 
In Section~\ref{sec:slicing}, we determine the value of $s(3, k)$ for all values of $k$.
Namely, $s(3, 2)=4$ and $s(3, k)=2k-1$ for $k\geq 3$. 
We remark that all our constructions were found by hand without the help of computers.


\section{$k$-piercing brick partitions}\label{sec:piercing}

In this section, we prove Theorem~\ref{thm:piercing} by exhibiting an explicit construction of a $k$-piercing brick partition of a $3$-dimensional brick for an integer  $k\geq 3$. 
Assume a $3$-dimensional brick $B$ is given. 
By scaling the sides of $B$, we may assume $B=\iii{0}{6}{0}{6}{0}{6}$ for convenience. 

Consider the set of bricks $\bp=\{W_1, W_2, W_3\}\cup\{X_i, X'_i, Y_i, Y'_i, Z_i, Z'_i: i\in\{1,2\}\}$ where each brick in $\bp$ is defined as below (See Figure~\ref{fig:brick3d-p}):

\begin{minipage}[t]{.3\textwidth} 
\begin{itemize}
\item $X_1=\iii{0}{2}{3}{6}{0}{4}$
\item $X_2=\iii{4}{6}{0}{3}{2}{6}$
\item $X'_1=\iii{3}{4}{2}{6}{4}{6}$
\item $X'_2=\iii{2}{3}{0}{4}{0}{2}$
\item $W_1=\iii{0}{2}{0}{2}{0}{2}$
\end{itemize}
\end{minipage}
\begin{minipage}[t]{.3\textwidth} 
\begin{itemize}
\item $Y_1=\iii{0}{4}{0}{2}{3}{6}$
\item $Y_2=\iii{2}{6}{4}{6}{0}{3}$
\item $Y'_1=\iii{0}{2}{2}{3}{0}{4}$
\item $Y'_2=\iii{4}{6}{3}{4}{2}{6}$
\item $W_2=\iii{2}{4}{2}{4}{2}{4}$
\end{itemize}
\end{minipage}
\begin{minipage}[t]{.3\textwidth} 
\begin{itemize}
\item $Z_1=\iii{0}{3}{2}{6}{4}{6}$
\item $Z_2=\iii{3}{6}{0}{4}{0}{2}$
\item $Z'_1=\iii{0}{4}{0}{2}{2}{3}$
\item $Z'_2=\iii{2}{6}{4}{6}{3}{4}$
\item $W_3=\iii{4}{6}{4}{6}{4}{6}$
\end{itemize}
\end{minipage}\\

\begin{figure}[h!]
  \centering
  \includegraphics[width=14cm,page=1]{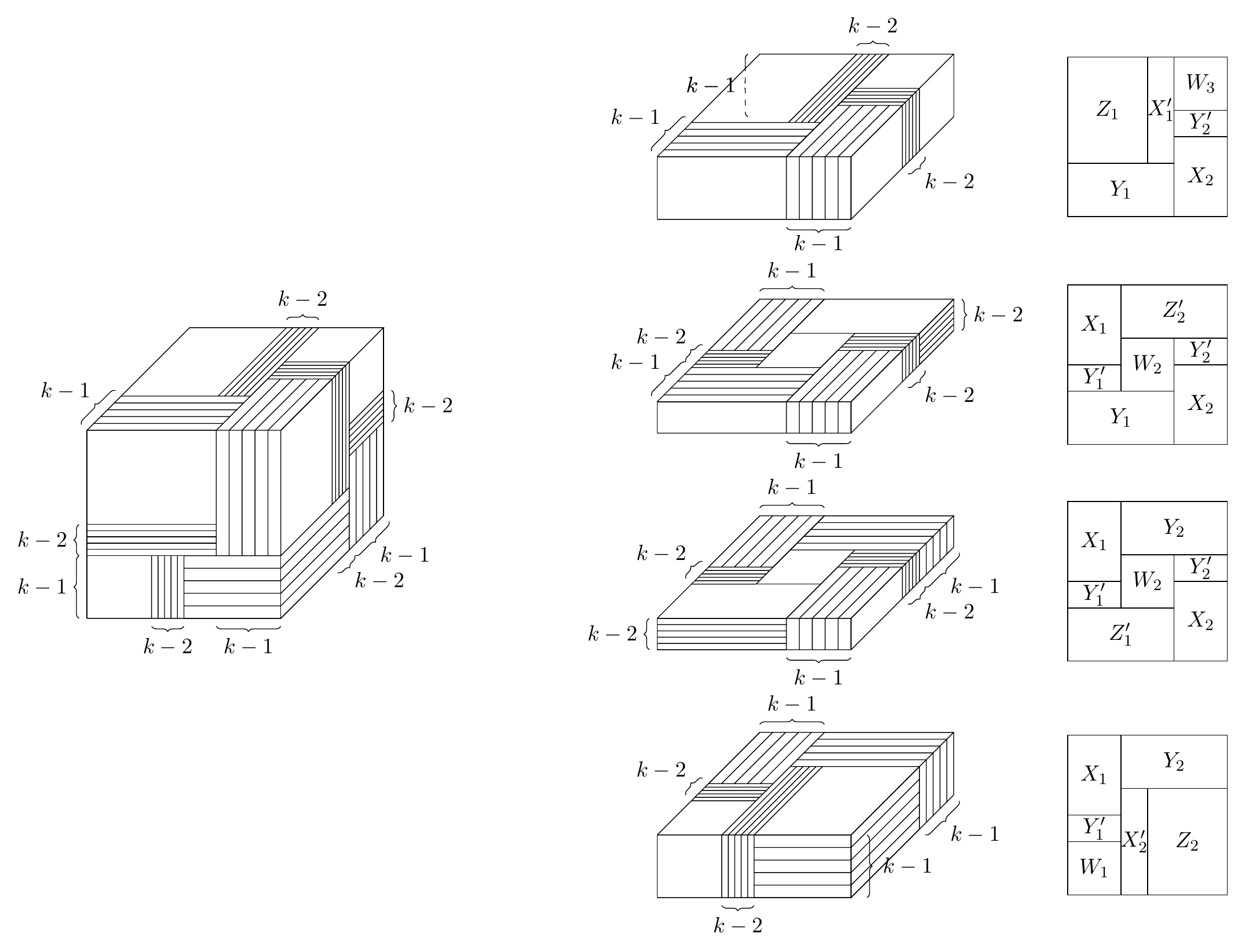}
  \caption{A $k$-piercing brick partition of a $3$-dimensional brick with only $12k-15$ bricks.}
  \label{fig:brick3d-p}
\end{figure}

\begin{clm}
$\bp$ is a brick partition of $B$. 
\end{clm}
\begin{proof}
It is not hard to see that $W_1, W_2, W_3$ together with 
\begin{center}
\hfill $\iii{0}{2}{2}{6}{0}{4}$\hfill$\iii{0}{4}{0}{2}{2}{6}$\hfill $\iii{0}{4}{2}{6}{4}{6}$\hfill \,

\hfill $\iii{4}{6}{0}{4}{2}{6}$\hfill $\iii{2}{6}{4}{6}{0}{4}$\hfill $\iii{2}{6}{0}{4}{0}{2}$\hfill\,
\end{center}
 is a brick partition of $B$. 
 Each brick that is not $W_i$ can be refined into two bricks in $\bp$ as follows:
 \begin{center}
\hfill $X_1\sqcup Y'_1=\iii{0}{2}{2}{6}{0}{4}$
\hfill $Y_1\sqcup Z'_1=\iii{0}{4}{0}{2}{2}{6}$
\hfill $Z_1\sqcup X'_1=\iii{0}{4}{2}{6}{4}{6}$
\hfill \,

\hfill $X_2\sqcup Y'_2=\iii{4}{6}{0}{4}{2}{6}$
\hfill $Y_2\sqcup Z'_2=\iii{2}{6}{4}{6}{0}{4}$
\hfill $Z_2\sqcup X'_2=\iii{2}{6}{0}{4}{0}{2}$
\hfill \,
\end{center}
Since every brick in $\bp$ was used exactly once, we have shown that $\bp$ is a brick partition of $B$.
\end{proof}

\noindent
We now refine all bricks in $\bp$ except $W_1, W_2, W_3$ to obtain a new brick partition $\bp^*$ of $B$ according to Table~\ref{table:refine}. 
Each line of the table specifies which brick is partitioned into how many pieces along which direction; 
imagine a plane orthogonal to the specified direction slicing the brick into smaller bricks as it moves along the specified direction.
For instance, the first line of Table~\ref{table:refine} means the following: partition $X_1$ into $k-1$ bricks by cutting along the first axis. 
See Figure~\ref{fig:brick3d-p}.
It is clear that $\bp^*$ is a brick partition of $B$ since $\bp$ is a brick partition of $B$. 
Note that $\bp^*$ has $3+6(k-1)+6(k-2)=12k-15$ bricks.

\begin{table}[h]
\begin{minipage}[b]{0.45\linewidth}
\centering
\begin{tabular}{c||c|c}
brick & cut into in $\bp^*$ & direction \\
\hline
\hline
 $X_1$ & $k-1$ & $1$ \\
 $X_2$ & $k-1$ & $1$ \\
 $X'_1$ & $k-2$ & $1$ \\
 $X'_2$ & $k-2$ & $1$ \\
\hline
 $Y_1$ & $k-1$ & $2$ \\
 $Y_2$ & $k-1$ & $2$ \\
 $Y'_1$ & $k-2$ & $2$ \\
 $Y'_2$ & $k-2$ & $2$ \\
\hline
 $Z_1$ & $k-1$ & $3$ \\
 $Z_2$ & $k-1$ & $3$ \\
 $Z'_1$ & $k-2$ & $3$ \\
 $Z'_2$ & $k-2$ & $3$ \\
\end{tabular}        
	\caption{Forming $\bp^*$ from $\bp$.}
       \label{table:refine}
\end{minipage}
\hspace{0.5cm}
\begin{minipage}[b]{0.45\linewidth}
\centering
\begin{tabular}{c||c|c}
direction & line & intersecting brick in $\bp$ \\
\hline
\hline
$1$ & $\ell(*, [3,6], [0,4])$ & $X_1$   \\
$1$ & $\ell(*, [0,3], [2,6])$ & $X_2$   \\
$1$ & $\ell(*, [2,6], [4,6])$ & $X'_1$  \\
$1$ & $\ell(*, [0,4], [0,2])$ & $X'_2$  \\
\hline
$2$ & $\ell([0,4], *, [3,6])$ & $Y_1$   \\
$2$ & $\ell([2,6], *, [0,3])$ & $Y_2$   \\
$2$ & $\ell([0,2], *, [0,4])$ & $Y'_1$  \\
$2$ & $\ell([4,6], *, [2,6])$ & $Y'_2$ \\
\hline
$3$ & $\ell([0,3], [2,6], *)$ & $Z_1$   \\
$3$ & $\ell([3,6], [0,4], *)$ & $Z_2$   \\
$3$ & $\ell([0,4], [0,2], *)$ & $Z'_1$  \\
$3$ & $\ell([2,6], [4,6], *)$ & $Z'_2$  
\end{tabular}
        \caption{Lines intersecting bricks.}
        \label{table:proof}
\end{minipage}
\end{table}

\begin{clm}
$\bp^*$ is a $k$-piercing brick partition of $B$. 
\end{clm}
\begin{proof}
For intervals $I, J$, let $\ell(*, I, J)$ be the set of lines parallel to the first axis where the second coordinate and third coordinate are fixed values in $I$ and $J$, respectively. 
Similarly, define $\ell(I, *, J)$ and $\ell( I, J, *)$ to be sets of lines parallel to the second and third, respectively, axis where the other coordinates are fixed values in $I$ and $J$. 

Table~\ref{table:proof} conveys that each axis-parallel line intersects some brick in $\bp$ that is refined in $\bp^*$. 
For instance, the first line of Table~\ref{table:proof} means the following: 
each line $\ell_1$ in $\ell(*, [3,6], [0,4])$ intersects $X_1$, which is cut into $k-1$ bricks in $\bp^*$ (see Table~\ref{table:refine}).
Since $\bp^*$ is a partition, $\ell$ intersects at least one more brick in $\bp^*$, for a total of at least $k$ intersecting bricks in $\bp^*$. 
Note that every axis-parallel line is in one of the sets defined in Table~\ref{table:proof}. 
It is easy to check Table~\ref{table:proof} to see that each axis-parallel line intersects at least $k$ bricks in $\bp^*$. 
\end{proof}
 

\section{$k$-slicing brick partitions}\label{sec:slicing}

In this section, we prove Theorem~\ref{thm:slicing}.
Throughout this section, let $B$ be a $3$-dimensional brick. 
By scaling the sides of $B$, we may assume $B=\iii{0}{2}{0}{2}{0}{2}$ for convenience. 
We first prove a lower bound on $s(3, k)$. 

\begin{prop}
For an integer $k\geq2$, $s(3, k)\geq 2k-1$.
\end{prop}
\begin{proof}
Let $\bp$ be an arbitrary $k$-slicing brick partition of $B$.
For each brick $b\in \bp$, let $f(b)$ be the number of boundary planes of $B$ incident with $b$. 
Let $F(\bp)=\sum_{b\in\bp} f(b)$.
Since $\bp$ is $k$-slicing, each of the six boundary planes of $B$ meets at least $k$ bricks of $\bp$, so $F(\bp)\geq 6k$. 

On the other hand, a brick $b\in\bp$ is incident with at most four boundary planes of $B$, since otherwise there is an axis-parallel plane that meets only one brick, namely, $b$. 
Moreover, if $b$ is incident with exactly four boundary planes, then it must contain two corners of $B$, since otherwise there is an axis-parallel plane that meets only one brick, namely, $b$.
Thus, at most four bricks in $\bp$ are incident with exactly four boundary planes. 
Thus, $F(\bp)\leq 4\alpha+3(|\bp|-\alpha)$, where $\bp$ has exactly $\alpha$ bricks incident with exactly four boundary planes of $B$. 

Therefore, $6k\leq \alpha+3|\bp|$, which implies $|\bp|\geq 2k-\alpha/3\geq 2k-4/3$. 
Since $|\bp|$ is an integer, we conclude $|\bp|\geq 2k-1$. 
\end{proof}

When $k=2$, the lower bound on $s(3, k)$ given by the previous proposition is not tight. 
We now determine the value of $s(3, 2)$. 
The arguments have similar flavor to the one in the proof of the previous proposition.

\begin{prop}
$s(3, 2)=4$.
\end{prop}
\begin{proof}
We first claim that every $2$-slicing brick partition of $B$ has at least four bricks. 
Let $C$ be the set of the following corners $(0, 0, 0), (0, 2, 2), (2, 0, 2), (2, 2, 0)$.
Given an arbitrary $2$-slicing brick partition $\bp_1$ of $B$, if a brick in $\bp_1$ contains two corners $c_1, c_2\in C$, then the unique axis-parallel plane containing both $c_1, c_2$ does not meet two bricks of $\bp_1$. 
Thus, each brick in $\bp_1$ contains at most one corner in $C$, so $\bp_1$ has at least four bricks. 

Consider the set of bricks $\bp_2=\{X_0, X_1, Y_0, Y_1\}$ where each brick in $\bp_2$ is defined as below (See Figure~\ref{fig:brick3d-s}):

\begin{minipage}[t]{.5\textwidth} 
\begin{itemize}
\item $X_0=\iii{0}{2}{0}{1}{0}{1}$
\item $X_1=\iii{0}{2}{1}{2}{0}{1}$
\end{itemize}
\end{minipage}
\begin{minipage}[t]{.5\textwidth} 
\begin{itemize}
\item $Y_0=\iii{0}{1}{0}{2}{1}{2}$
\item $Y_1=\iii{1}{2}{0}{2}{1}{2}$
\end{itemize}
\end{minipage}\\

It is not hard to see that $\bp_2$ is a brick partition of $B$, and moreover, $\bp_2$ is $2$-slicing.
\end{proof}

\begin{figure}[h!]
  \centering
  \includegraphics[height=6cm,page=1]{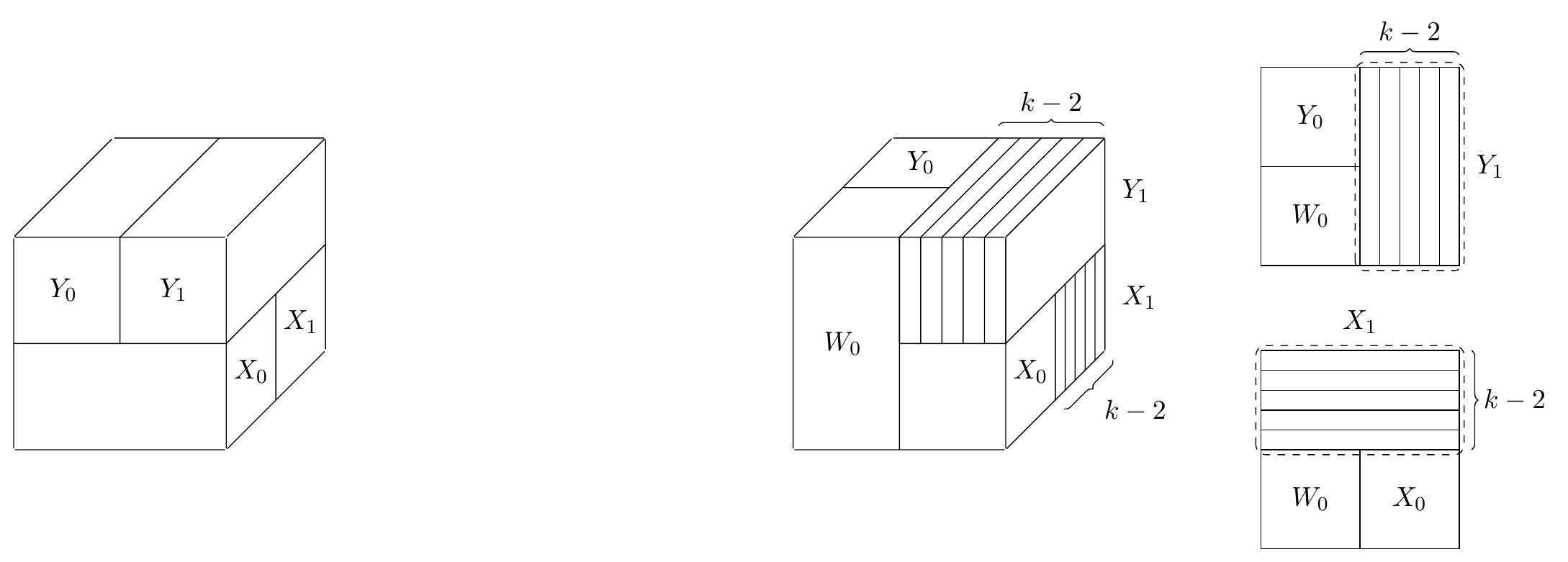}
  \caption{A $k$-slicing brick partition of a $3$-dimensional brick for $k=2$ (left) and for $k\geq 3$ (right).}
  \label{fig:brick3d-s}
\end{figure}

Now consider $k\geq 3$. 
We exhibit an explicit construction of a $k$-slicing brick partition of $B$. 
Consider the set of bricks $\bp=\{W_0, X_0, X_1, Y_0, Y_1\}$ where each brick in $\bp$ is defined as below (see Figure~\ref{fig:brick3d-s}.):

\begin{minipage}[t]{.3\textwidth} 
\begin{itemize}
\item $W_0=\iii{0}{1}{0}{1}{0}{2}$
\end{itemize}
\end{minipage}
\begin{minipage}[t]{.3\textwidth} 
\begin{itemize}
\item $X_0=\iii{1}{2}{0}{1}{0}{1}$
\item $X_1=\iii{0}{2}{1}{2}{0}{1}$
\end{itemize}
\end{minipage}
\begin{minipage}[t]{.3\textwidth} 
\begin{itemize}
\item $Y_0=\iii{0}{1}{1}{2}{1}{2}$
\item $Y_1=\iii{1}{2}{0}{2}{1}{2}$
\end{itemize}
\end{minipage}\\

\begin{clm}
$\bp$ is a brick partition of $B$. 
\end{clm}
\begin{proof}
It is not hard to see that $X_0, X_1, Y_0, Y_1$ together with 
\begin{center}
\hfill $\iii{0}{1}{0}{1}{0}{1}$\hfill$\iii{0}{1}{0}{1}{1}{2}$\hfill \,
\end{center}
 is a brick partition of $B$. 
The above two bricks together is a partition of $W_0$.
Since every brick in $\bp$ was used exactly once, we have shown that $\bp$ is a brick partition of $B$.
\end{proof}

\noindent
We now refine the bricks $X_1$ and $X_2$ to obtain a new brick partition $\bp^*$ of $B$.
Recall that ``cutting along'' means a plane orthogonal to the specified direction is slicing the brick into smaller bricks as it moves along the specified direction.

\noindent 
-- Partition $X_1$ into $k-2$ bricks by cutting along the second axis.\\ 
-- Partition $Y_1$ into $k-2$ bricks by cutting along the first axis. 

\noindent
It is clear that $\bp^*$ is a brick partition of $B$ since $\bp$ is a brick partition of $B$. 
Note that $\bp^*$ has $3+2(k-2)=2k-1$ bricks. 

\begin{clm}
$\bp^*$ is a $k$-slicing brick partition of $B$. 
\end{clm}
\begin{proof}
For an interval $I$, let $\mathcal L(I, *, *)$ be the set of planes whose normal vectors are parallel to the first axis where the first coordinate is a fixed value in $I$. 
Similarly, define $\mathcal L( *, I, *)$ and $\mathcal L(*, *, I)$ to be sets of planes whose normal vectors are parallel to the second and third axes where the second and third, respectively, coordinates are fixed values in $I$. 

Each plane in $\mathcal L(*, *, I)$ intersects $W_0$ and either $X_0, X_1$ or $Y_0, Y_1$ in $\bp$, so it intersects $k$ bricks in $\bp^*$.
Each plane in $\mathcal L(I, *, *)$ intersects $X_1$, either $W_0$ or $X_0$, and either $Y_0$ or $Y_1$ in $\bp$, so it intersects $k$ bricks in $\bp^*$.
Each plane in $\mathcal L(*, I, *)$ intersects $Y_1$, either $W_0$ or $Y_0$, and either $X_0$ or $X_1$ in $\bp$, so it intersects $k$ bricks in $\bp^*$.
\end{proof}

\section*{Acknowledgments}

The first author thanks Jinha Kim for directing his attention to~\cite{2019BuLiLoWa}. 
The authors thank the anonymous referees for useful comments that led to improvements on the manuscript. 
This research has been composed with the support of Gyeonggi Science High School Autonomous Research of the year 2019.

\bibliography{ref}{}         

\begin{thebibliography}{1}

\bibitem{2005AhYu}
R.~Ahlswede and A.~A. Yudin.
\newblock {\em On Partitions of a Rectangle into Rectangles with Restricted
  Number of Cross Sections}, pages 941--954.
\newblock Springer Berlin Heidelberg, Berlin, Heidelberg, 2006.

\bibitem{2002AlBoHoKl}
Noga Alon, Tom Bohman, Ron Holzman, and Daniel~J. Kleitman.
\newblock On partitions of discrete boxes.
\newblock {\em Discrete Mathematics}, 257(2):255 -- 258, 2002.
\newblock Kleitman and Combinatorics: A Celebration.

\bibitem{2019BuLiLoWa}
Matija Bucic, Bernard Lidick\'{y}, Jason Long, and Adam~Zsolt Wagner.
\newblock Partition problems in high dimensional boxes.
\newblock {\em J. Combin. Theory Ser. A}, 166:315--336, 2019.

\bibitem{2004GrKiPr}
Jaros{\l}aw Grytczuk, Andrzej~P. Kisielewicz, and Krzysztof Przes{\l}awski.
\newblock Minimal partitions of a box into boxes.
\newblock {\em Combinatorica}, 24(4):605--614, 2004.

\bibitem{2019Holzman}
Ron Holzman.
\newblock On 2-colored graphs and partitions of boxes.
\newblock {\em European J. Combin.}, 79:214--221, 2019.

\bibitem{2006KiPr}
Andrzej~P. Kisielewicz and Krzysztof Przes{\l}awski.
\newblock On the number of minimal partitions of a box into boxes.
\newblock {\em Discrete Math.}, 306(8-9):843--846, 2006.

\bibitem{2008KiPr}
Andrzej~P. Kisielewicz and Krzysztof Przes{\l}awski.
\newblock Polyboxes, cube tilings and rigidity.
\newblock {\em Discrete Comput. Geom.}, 40(1):1--30, 2008.

\bibitem{1993Kleitman}
Daniel~J. Kleitman.
\newblock Partitioning a rectangle into many subrectangles so that a line can
  meet only a few.
\newblock In {\em Planar graphs ({N}ew {B}runswick, {NJ}, 1991)}, volume~9 of
  {\em DIMACS Ser. Discrete Math. Theoret. Comput. Sci.}, pages 95--107. Amer.
  Math. Soc., Providence, RI, 1993.

\end{thebibliography}
\bibliographystyle{plain}  

\end{document}